\documentclass[preprint,11pt]{elsarticle}

\usepackage{amsmath}
\usepackage{amssymb}
\usepackage[english]{babel}
\usepackage{amsthm}
\usepackage{verbatim}
\usepackage{graphicx}
\usepackage{caption}
\usepackage{subcaption}
\usepackage{float}
\usepackage{placeins}
\usepackage{multicol}
\usepackage{xcolor}
\usepackage{fullpage}

\title{Optimality and Sustainability of Delayed Impulsive Harvesting}

\author{Jennifer Lawson and Elena Braverman 
\\
 Dept. of Math \& Stats, University of Calgary, \\
  2500 University Dr. NW, 
	Calgary, AB, Canada T2N1N4 }
	
\date{Final version}	

\journal{Communications in Nonlinear Science and Numerical Simulation}

\newtheorem{theorem}{Theorem}[section]
\newtheorem{corollary}{Corollary}[theorem]
\newtheorem{lemma}[theorem]{Lemma}

\theoremstyle{definition}

\newtheorem{remark}{Remark}

\theoremstyle{definition}

\begin{document}

\begin{abstract}
We consider  a logistic differential equation subject to impulsive delayed harvesting, where the deduction information is a function of the population size at the time of one of the previous impulses.
A close connection to the dynamics of high-order difference equations is used to conclude
that while the inclusion of a delay in the impulsive condition does not impact the optimality of the yield, sustainability may be highly affected and is generally delay-dependent. Maximal and other types of yields are explored, and sharp stability tests are obtained for the model, as well as explicit sufficient conditions. It is also shown that persistence of the solution is not guaranteed for all positive initial conditions, and extinction in finite time is possible, as is illustrated in the simulations. 

{\bf Keywords:} Optimal harvesting; logistic equation; impulsive system; impulsive delayed harvesting; population dynamics

{\bf AMS (MOS) subject classification:}  92D25, 34A37 

\end{abstract}

\maketitle

\section{Introduction}

Optimal sustainable resource management is one of the most important modern problems, with the dynamics of harvesting models being an essential facet \cite{Clark1990}.  Harvesting can be represented in models either continuously or as only occurring during short-time periods.  Continuous harvesting can be described as a continuous deduction term appearing in an ordinary Differential Equation (DE) determining population dynamics. It assumes that harvesting occurs without any interruptions, whereas impulsive harvesting corresponds to part of the stock being removed at specific moments in time, with the duration of the harvesting event being negligible compared to the process time. Even though continuous harvesting may be preferable from the point of view of both maximizing harvest and sustainability \cite{Mamdani2008,xiao}, it is not always realistic or easily applicable. This is why investigation of impulsive harvesting models is important.

Impulsive DEs incorporate two parts, the DE that describes the behaviour of the system during times of continuous dynamics, and the conditions that govern the instantaneous change in the system at the impulsive moments. Usually, this instantaneous change is a result of some external effect on the system, which has a duration that is negligible compared to the overall time scale of the process. Impulsive DEs have many practical applications such as pest control \cite{Smith?2019}, pulse vaccination strategies \cite{GaoChenTeng2007}, and optimal harvesting in fisheries \cite{CordovaLepe2012}. For more on the theory of impulsive DEs see the monograph \cite{BainovSimeonov1993}.

It is well known that including delays within a DE model of population dynamics can lead to major changes to its behaviour, such as causing instability, oscillations, and extinction,
which are not observed in a corresponding ordinary DE model \cite{Smith2011}. The famous Hutchinson equation is one such example. It can be compared to the logistic equation where the carrying capacity is always a globally attractive equilibrium for all non-trivial positive solutions. The inclusion of a delay in the Hutchinson equation can cause the carrying capacity equilibrium to become unstable for certain values of a delay.
Other examples specific to harvesting models can be seen in \cite{Kar2003,Xu2021}, where delay is incorporated into the continuous harvesting terms.

In an impulsive system, whenever control is involved, it is natural to assume that the information available at the control implementation point is not up-to-date, leading to delayed impulsive conditions. 
The incorporation of delays in impulses goes back to the 1990s~\cite{Akca1996} with the further development of non-instantaneous impulse theory, some of which is summarized in
the monographs~\cite{Agarwal2017,Liu2019}, and with recent progress reported in  \cite{Church2022,Church2017,Church2019}. Delayed impulses in the context of harvesting were explored in \cite{Pei2006}.
To the best of our knowledge, there still has been no investigation of the optimal harvesting policies and their sustainability for systems with delayed impulsive harvesting in the literature, and we aim to fill this gap.

For logistic and other simple population models, such as Gompertz, incorporating stochastic fluctuations or random differential equations created an additional challenge but reflected unpredictable changes in the environment, see \cite{Cortes2021,new_Wang,new_Yang} and the references therein. 
Impulsive harvesting of a stochastic Gompertz model was a focus of \cite{new_Wang}.
Though logistic-type equations are considered in most studies on optimal control, the use of other growth rates can modify the preferred policies \cite{new_Cruz-Rivera}. 
The close connection of continuous models to difference equations has led to extensive study of discrete population models \cite{BC_monograph} including harvesting \cite{new_Grey}. The fact that species do not naturally exist in isolation but coexist, compete or serve as prey for others, led to extensive literature on harvesting of a single or multiple populations in a food chain, and on optimal yields for exploited species \cite{new_Dawed,new_Demir,new_Pei,new_Upadhyay}. 
Incorporating harvesting in systems of differential or difference equations includes the case of structured populations where selective harvesting is allowed \cite{new_Jana,new_Skonhoft}. This approach is quite useful in policy-making, for example, when only fish types within a certain size range are eligible to take-home vs catch-and-release policies.

We consider the logistic equation with constant catch-per-unit effort impulsive harvesting that is dependent upon delayed data.  
This assumes that the information used to determine hunting or fishery quotas is based on data for the population size or structure which was collected during one of the previous harvesting events. 

The main object of the present paper is the delayed impulsive harvesting model, given for a fixed $k\in\mathbb{N}$ as
\begin{equation}
\begin{cases}
\label{Impulse:Delay}
\frac{\displaystyle dN}{\displaystyle dt} = r N(t) \bigg(1 - \frac{N(t)}{K_c} \bigg), & t\neq nT,\text{ } n\in\mathbb{N}\\
 N(nT^+) = \max\{N(nT) - E N((n -  k)T),0\}, & t = nT, \text{ }n\in\mathbb{N}\text{ }\\
N(0^+)=N_0^+, N(0) = N_0,...,N(-(k-1)T) = N_{-(k-1)}\\
\end{cases}
\end{equation}
with prescribed initial conditions
\begin{equation*}
N_0^+, N_{i} \in (0,\infty), i=-(k-1),...,0.
\end{equation*}
In this model, the left-continuous $N(t)$ represents a size or a biomass of the population as a function of time, 
$r>0$ is the intrinsic growth rate, $K_c>0$ is a carrying capacity of the environment, 
$T > 0$ is the time between two consecutive harvesting events, $E$ is a harvesting effort and is assumed to be $E\in(0,1)$ to avoid immediate extinction.
We assume both that restocking does not occur, and that the function $N(t)$ is left continuous. We denote the size of the population after harvesting as $\lim\limits_{h\to 0^+} N(nT + h) = N(nT^+)$, whereas the size of the population $N(t)$ before harvesting is $N(nT)$, for the left-continuous function $N$.
 Note that we do not require continuity at zero i.e. $N(0^+)=N(0)$, nor do we assume anything about the size of the population between $t=-iT$ and $t=-(i-1)T$. This makes room for the possibility that the first harvesting event occurs at $t=0$, as might be likely in a real-world setting, but does not require that this happens.
We assume $k\in \mathbb{N}$, though references to the non-delay model with $k=0$ will also be given. 

The motivation of incorporating delay in the impulsive harvesting is two-fold:
\begin{enumerate}
\item
Impulsive harvesting in general allows us to describe short duration harvesting events, the effect of which makes it impossible for population growth to remain continuously at optimal levels. It is a well developed topic, with multiple publications appearing in the literature \cite{Mamdani2008,CordovaLepe2012,Pei2006,xiao,Zhang2003}, but not for delayed
impulsive conditions.
Thus, delayed impulsive models are of theoretical interest. While the consideration of delayed impulsive systems is not new \cite{Akca1996}, recent developments in the theory of such dynamical systems \cite{Church2022,Church2017,Church2019} have made its practical study feasible, as there are theoretical tools to investigate it.
 
\item
Harvesting policies are dependent upon population data, and often the data which is used is not up-to-date, leading to a delay in the harvesting term of models. It is therefore important to be able to assess the impact of the delay, so that managers may determine what additional modifications must be made to the harvesting decisions. While non-delay 
harvesting models depict gradual declines to extinction, they mostly failed to explain a frequently observed phenomenon of instantaneous species disappearance due to over-exploitation.
Unlike traditional continuous or non-delayed impulsive harvesting, delayed impulsive harvesting can describe {\em immediate}, not long-term collapse of harvested population. Truncated models where the population or commodity size is chosen as a maximum of the computed value or zero, are quite common in mathematical economics and discrete dynamical systems. We found it natural to extend this method to population dynamics in order to describe possible extinction in finite time.
\end{enumerate}

Let us also dwell on the choice of the model. Our purpose was to explore and emphasize the effect of the harvesting delay in the impulsive form, thus we consider the simplest autonomous logistic equations. The results will outline the intrinsic effect of the delay in short-term harvesting.

Our goal is first, to consider the sustainability of \eqref{Impulse:Delay} under harvesting, which corresponds to the local asymptotic stability of a positive solution which will be described later, and second, to explore the sustainable yield (SY) and the maximum sustainable yield (MSY) of 
\eqref{Impulse:Delay}. The paper is structured to follow this purpose. After presenting relevant definitions and auxiliary results in Section~\ref{sec:prelim}, we explore stability of \eqref{Impulse:Delay} in Section~\ref{sec:stabil}. All the issues connected to SY and MSY, and relevant solutions of \eqref{Impulse:Delay}, are postponed to Section~\ref{sec:MSY}.
We show that, while optimality is unaffected by the magnitude of delay, sustainability of the optimal solution is delay-dependent for $k \geq 2$. The analysis of the impact of the delay on local asymptotic stability of the positive solution is based on the results obtained in Section~\ref{sec:stabil}. Finally, Section~\ref{sec:concl} includes examples, numerical simulations, as well as discussion of the results of the paper and possible future directions.

\section{Preliminaries and Auxiliary Results}
\label{sec:prelim}

A solution $N^*(t)$ of impulsive system~\eqref{Impulse:Delay} is said to be {\em stable} if for any $\varepsilon > 0$ there exists a $\delta = \delta(\varepsilon)>0$ such that the inequalities $|N_j - N^*(jT) |< \delta$, $j=-k+1, \dots, 0$, $|N_0^+ - N^*(0^+) | < \delta$
imply $|N(t) - N^*(t) |<\varepsilon$ for all $t>0$. 
A solution of \eqref{Impulse:Delay} is {\em (locally) asymptotically stable} if it is stable and there exists $\eta>0$ such that $\lim\limits_{t\to\infty} |N(t) - N^*(t)| = 0$ for any $|N_j - N^*(jT) |< \eta$, $j=-k+1, \dots, 0$, $|N_0^+ - N^*(0^+) | < \eta$.
A solution of \eqref{Impulse:Delay} is  {\em globally asymptotically stable} if it is stable and  $\lim\limits_{t\to\infty} |N(t) - N^*(t)| = 0$ for any $N_0^+>0, N_0, N_{-1}..., N_{-k+1} \geq 0$.

A $k$-th order difference equation has the form
\begin{equation}
\label{Difference:General}
x_{n+1} = f(x_n,x_{n-1}, ... , x_{n-k}), \quad n\in {\mathbb N}, \quad n \geq k
\end{equation}
with the initial conditions $x_0,..., x_{k}$.

A solution $x_n \equiv x^*$ of \eqref{Difference:General} is {\em stable} if for all $\varepsilon>0$, there exists a $\delta >0$ such that $\max\{ |x_0 - x^*|,...,|x_{k}-x^*|\}< \delta$ implies $|x_n - x^*|<\varepsilon$ for any $n \geq k$.
A solution $x^*$ of \eqref{Difference:General} is {\em (locally) asymptotically stable} if it is stable and there exists $\eta>0$ such that $\max\{ |x_0 - x^*|,...,|x_{k}-x^*|\}<\eta$ implies $\lim\limits_{n\to \infty}| x_n - x^*|= 0$.
A solution $x^*$ of \eqref{Difference:General} is {\em globally asymptotically stable} if it is stable and $\lim\limits_{n\to \infty}| x_n - x^*|= 0$ for any $x_0, ..., x_{k}$.

If harvesting is restricted to only the surplus production of a population, then theoretically, harvesting should be able to continue indefinitely without drastically altering the stock levels. The idea behind the Maximum Yield (MY) is that it corresponds to an optimal solution of \eqref{Impulse:Delay} such that the yield will not be exceeded by any other solution, and that the optimal solution is $T$-periodic leading to a constant yield over a time period.
A Maximum Sustainable Yield (MSY) is a MY where the optimal solution of \eqref{Impulse:Delay} is (at least locally) asymptotically stable.

In \cite{Zhang2003}, the authors considered an MSY for \eqref{Impulse:Delay} with $k=0$.
The results are summarized in the following.

\begin{lemma} \cite{Zhang2003}
\label{ZhangThm}
Consider \eqref{Impulse:Delay} for $k=0$
\begin{equation}
\label{Impulse:NoDelay}
\begin{cases}
\displaystyle \frac{dN}{dt} = r N(t) \bigg(1 - \frac{N(t)}{K_c} \bigg), \text{ } t\neq n T,\quad n\in\mathbb{N}\\
N(nT^+) = \max\{(1-E) N(nT),0\},\text{ }t = n T, \quad n\in \mathbb{N}\\
N(0) = N_0.
\end{cases}
\end{equation}
Then the optimal harvesting effort is
\begin{equation}
\label{E_opt}
E_{opt} = 1 - e^{-rT/2},
\end{equation}
and the MSY is given by
\begin{equation}
\label{MSY}
MSY = \frac{K_c(e^{rT/2} - 1)}{T(e^{rT/2}+1)}
\end{equation}
The optimal positive periodic solution $N^*(t)$ of \eqref{Impulse:NoDelay} corresponding to the MSY and $E_{opt}$ satisfies
\begin{equation}
\label{eq:optsol}
N^*(nT^+) = \frac{K_c}{e^{rT/2}+1}
\end{equation}
and is globally asymptotically stable.
\end{lemma}

In \cite{Zhang2003} analysis of a non-delayed impulsive model \eqref{Impulse:NoDelay} is reduced to a nonlinear difference equation of the first order. We also intensively exploit connection between difference and impulsive equations. 
When a delay is incorporated in impulsive condition, the difference equation becomes higher order. 
We recall that for difference equations, the roots of the characteristic equation of an associated linearized model should lie inside the unit circle for local asymptotic stability, in contrast to differential equations where the real parts of the roots have to be negative. Some auxiliary results  regarding difference equations are listed below.

\begin{lemma} [\cite{Elaydi2005}]
\label{Cor:CritK1}
The roots of the characteristic equation  
\begin{equation*}
p(\lambda) = \lambda^2 - p_0 \lambda + p_1, \quad p_0>0, \quad p_1>0
\end{equation*}
lie inside the unit circle if and only if $p_0-1 < p_1 < 1$.
\end{lemma}

The result of \cite[Theorem 1.1.1, Part f, P. 7]{KulenovicLadas2002} describes conditions when 
a root of a quadratic equation lies on the boundary of the unit circle.

\begin{lemma}[\cite{KulenovicLadas2002}]
\label{Lemma:OnUnitDisk}
A necessary and sufficient condition for a root of the characteristic equation 
\begin{equation*}
\lambda^2 - p_0 \lambda + p_1 = 0
\end{equation*}
with $p_0,p_1 \in {\mathbb R}$  to have a root satisfying $| \lambda | =1$ is that either 
\begin{equation*}
|p_0| = |1 + p_1| 
\end{equation*}
or
\begin{equation*}
p_1 = 1 \text{ and } |p_0|\leq 2.
\end{equation*}
\end{lemma}

Lemma~\ref{sufficient} is cited from \cite[Theorem 5.10, P.~253]{Elaydi2005}.

\begin{lemma}[\cite{Elaydi2005}]
\label{sufficient}
If $\displaystyle \sum_{i=0}^{k}|p_i|<1$ then the zero solution of the difference equation
\begin{equation*}
x_{n+k+1} +p_0 x_{n+k} + p_1 x_{n+k-1}+...+p_k x_n = 0
\end{equation*}
is asymptotically stable.
\end{lemma}

The following result can be found in \cite[Theorem 5.3, P.~249]{Elaydi2005}.

\begin{lemma}[\cite{Elaydi2005}]
\label{Lem:MoreGeneralk}
Let $p_0>0$, $p_k\in {\mathbb R}$ be arbitrary, and $k \in {\mathbb N}$. 
The zero solution of the equation
\begin{equation}
\label{geneq:difference}
x_{n+1} - p_0 x_n + p_k x_{n-k} = 0
\end{equation}
is asymptotically stable if and only if $|p_0| < (k+1)/k$
and
\\
(i) $\displaystyle |p_0| -1 < p_k < (p_0^2 + 1 - 2 |p_0| \cos(\theta^*))^{1/2}$ if $k$ is odd 
\\
or
\\
(ii) $|p_k-p_0| < 1$ and $\displaystyle |p_k| < (p_0^2 + 1 - 2 |p_0| \cos(\theta^*))^{1/2}$ if $k$ is even,
\\
where $\theta^*$ is the solution of the equation 
\begin{equation}
\label{theta}
\frac{\sin(k \theta)}{\sin((k+1)\theta)} 
= \frac{1}{|p_0|}, \quad \theta \in \left(0,\frac{\pi}{k+1} \right).
\end{equation}
\end{lemma}

However, we do not need the general form of Lemma~\ref{Lem:MoreGeneralk}, since in our model $0< p_k< p_0$. 
Then the left inequality in both (i) and (ii) becomes $p_0<p_k+1$, the right inequalities coincide.

\begin{corollary}
\label{Cor:MoreGeneralk}
Let $0< p_k< p_0$. Then \eqref{geneq:difference} is asymptotically stable 
if and only if the following two inequalities hold:
\begin{equation}
\label{gen_stab_cond}
p_0 < \min \left\{ p_k+1, \frac{k+1}{k} \right\}, \quad p_k < \sqrt{ p_0^2 + 1 - 2 p_0 \cos(\theta^*) } \, , 
\end{equation}
where $\theta^*$ is the solution of \eqref{theta}.
\end{corollary}

Lemma~\ref{Lem:Generalk} is cited from \cite[Theorem 5.2, P.~248]{Elaydi2005}.

\begin{lemma}[\cite{Elaydi2005}]
\label{Lem:Generalk}
Let $q \in (0,2)$. The zero solution of the equation
\begin{equation*}
x_{n+1} = x_n - q x_{n-k}
\end{equation*}
is asymptotically stable for $k=1$. 
It is asymptotically stable for $k \geq 2$ if and only if in addition
\begin{equation*}
q < 2 \cos\bigg( \frac{k \pi}{2k +1}\bigg).
\end{equation*}
\end{lemma}


\section{Stability with Delay Impulsive Harvesting}
\label{sec:stabil}

We start with reducing the dynamics of the differential equation with delayed impulsive harvesting to a difference equation. 

\begin{lemma}
\label{Lem:DiffSol}
The solution of \eqref{Impulse:Delay} 
on the interval $t\in(nT,(n+1)T)$, $n=0,1,\dots$ is 
\begin{equation}
\label{Impulse:Sol}
N(t) = \frac{K_c e^{r(t - nT)} N(nT^+)}{K_c + N(nT^+)(e^{r(t-nT)}-1)} \, .
\end{equation}
The solution of \eqref{Impulse:Delay} with $N(nT^+) = x_n$ at the points just after harvesting at $t = nT$ satisfies the difference equation
\begin{equation}
\label{Difference:k}
 x_{n+1} = \max\bigg\{\frac{K_c x_{n}e^{rT}}{K_c + x_{n}(e^{rT}-1)} - E \frac{K_c x_{n-k}e^{rT}}{K_c + x_{n-k}(e^{rT}-1)},0 \bigg\} = \max\bigg\{ f\left( x_n,x_{n-k} \right),0 \bigg\} , 
\end{equation}
where $n \geq k$. 
In addition $x_0 = N(0^+)$, and $x_1,...,x_{k}$ satisfy the relation
\begin{equation}
\label{Diff:Initial}
x_{i+1} = \max\bigg\{ \frac{K_c x_{i}e^{rT}}{K_c + x_{i}(e^{rT}-1)} - E N((i-k+1)T),0\bigg\}
\end{equation}
where $N((i-k+1)T)$ are given by the initial conditions for $i=0,...,k-1$.
\end{lemma}

\begin{proof}
Let $ N(nT^+) = x_n$ be the size of the population after a harvesting event. For $t\in(nT,(n+1)T)$, 
the impulsive model is non-delayed, and the solution of the differential equation exists, is monotone on the interval and is given by
\begin{equation*}
N(t) = \frac{K_c e^{r(t - nT)} N(nT^+)}{N(nT^+)(e^{r(t-nT)}-1)+K_c},\quad t\in(nT,(n+1)T).
\end{equation*}
The size of $N(t)$ at the end of the $n$-th time period before harvesting is 
\begin{equation*}
    N((n+1)T) = \frac{\displaystyle K_c N(nT^+) e^{rT}}{\displaystyle K_c + N(nT^+)(e^{rT}-1)}.
\end{equation*}
Using the value of $N(t)$ before harvesting as a function of $N(nT^+)$, combined with the definition of the impulsive condition in \eqref{Impulse:Delay}, leads to
\begin{equation*}
 N((n+1)T^+) =\max \left\{ \frac{\displaystyle K_c N(nT^+)e^{rT}}{\displaystyle K_c + N(nT^+)(e^{rT}-1)} - E N((n+1-k)T),0 \right\}.
\end{equation*}
For $n=0,...,k-1$, $N((n+1-k)T)$ are given by the initial conditions, leading to  relation \eqref{Diff:Initial}. For $n\geq k$ we once again treat $N((n+1-k)T)$ as a function of $N((n-k)T^+)$ and obtain
\begin{equation*}
 N((n+1)T^+) =\max \left\{ \frac{\displaystyle K_c N(nT^+)e^{rT}}{\displaystyle K_c + N(nT^+)(e^{rT}-1)} - E \frac{\displaystyle K_c N((n-k)T^+)e^{rT}}{\displaystyle K_c + N((n-k)T^+)(e^{rT}-1)},0 \right\} .
 \end{equation*}
Then, since $x_n = N(nT^+)$ the difference equation \eqref{Difference:k} is obtained.
\end{proof}

After the reduction to a difference equation, let us justify that stability of \eqref{Impulse:Delay} can be deduced from that of \eqref{Difference:k}.

\begin{lemma}
\label{connection}
The point $x^*$ is a locally asymptotically stable equilibrium of the difference equation \eqref{Difference:k} if and only if the solution to \eqref{Impulse:Delay} with $N^*(nT^+) = x^*$ is locally asymptotically stable.
\end{lemma}

\begin{proof}
Evidently local asymptotic stability of the solution to \eqref{Impulse:Delay} with $N(nT^+) = x^*$ implies local asymptotic stability of the equilibrium $x^*$ of \eqref{Difference:k}.

Further, let us assume that the solution $x^*$ of \eqref{Difference:k} is stable. 
Note that the function 
\begin{equation}
\label{F_def}
F(x,a) := \frac{K_c a x}{K_c+ x (a -1)}
\end{equation} 
for a fixed $a = e^{rs} >1$ and any non-negative $x$, has the derivative in $x$
$$
\frac{\partial F \left( x,e^{rs} \right)}{\partial x} = \frac{K_c^2 e^{rs}}{( K_c+ x (e^{rs} -1) )^2} > 0, \quad 
\left| \frac{\partial F \left( x,e^{rs} \right)}{\partial x} \right| \leq e^{rs} \leq e^{rT}.
$$
If the equilibrium $x^*$ of \eqref{Difference:k} is stable, for any $\varepsilon>0$ there exists $\delta>0$ such that, once all $|x_{j}-x^*|<\delta$, $j =0, \dots, k$, we get
$\displaystyle |x_n - x^*| < \varepsilon e^{-rT}$. 
The solution $N^*$ corresponding to $x^*$ on the interval $(nT, (n+1)T)$ is $N^*(nT^+)=x^*$ with 
$$
N^*(t) = \frac{K_c e^{r(t - nT)} x^*}{x^*(e^{r(t-nT)}-1)+K_c} \, .
$$
Following \eqref{Impulse:Sol}, we note that 
on $(nT,(n+1)T)$  there exist $\zeta$ between $x_n$ and $x^*$, $s \in [0,T]$ such that 
\begin{align*}
| N(t) - N^*(t) |  & = \left| \frac{K_c e^{r(t - nT)} x_n}{x_n(e^{r(t-nT)}-1)+K_c} - 
\frac{K_c e^{r(t - nT)} x^*}{x^*(e^{r(t-nT)}-1)+K_c} \right| \\ & = \left| \frac{\partial F(\zeta,^{rs})}{\partial x} \right| |x_n - x^*| \leq e^{rT} \varepsilon e^{-rT} = \varepsilon,
\end{align*}
thus $N^*$ is stable.

If the solution $x^*$ of \eqref{Difference:k} be locally asymptotically stable, and $n_0$ is such a number that
$\displaystyle |x_n - x^*| < \varepsilon e^{-rT}$ for $n \geq n_0$, as above,
$$
| N(t) - N^*(t) |  \leq e^{rT} \varepsilon e^{-rT} = \varepsilon, \quad t \geq nT, \quad n \geq n_0,
$$
thus $N^*$ is both stable and attractive, and therefore is locally asymptotically stable, which concludes the proof.
\end{proof}

Next, let us describe solution bounds for a harvested population.

\begin{lemma}
\label{lemma_bounds}
Let $E \in (0,1)$. Then for any non-negative initial values there exists $n_0 \geq k$ 
such that
the solution $x_n$ to \eqref{Difference:k} is in $[0,K_c]$ for $n \geq n_0$.
\end{lemma}

\begin{proof}
Let us note that if for some $n\geq k$, $x_n \leq K_c$, we also have $x_{n+1} \leq K_c$. Really, from monotone increasing of 
$F\left(x, a \right)$ defined in \eqref{F_def} in $x$ and $F\left(K_c, a \right)=K_c$ for any $a>1$, we get $x_{n+1} \leq F\left(x_n, e^{rT} \right) \leq K_c$. By induction, all $x_j \leq K_c$, $j \geq n$. Also, $F\left(x, a \right)> x$ for $x \in (0,K_c)$ and $F\left(x, a \right) < x$ for $x > K_c$.

Thus, we only have to exclude the case $x_n > K_c$ for any $n\geq k$.  
We have $x_n,x_{n-k} > K_c$,
$F\left(x_n, e^{rT} \right)<x_n$, $F\left(x_{n-k}, e^{rT} \right)> K_c$  and
$$
x_{n+1} = F\left(x_n, e^{rT} \right) - E F\left(x_{n-k}, e^{rT} \right)
< x_n - E K_c,
$$
therefore $x_{n+1}-x_n< - E K_c$, and after $\displaystyle j = \left\lfloor \frac{ x_n - K_c}{E K_c} \right\rfloor + 1$ steps, where $\lfloor y \rfloor$ is the integer part of $y$, we arrive at $x_{n+j} \leq K_c$, which concludes the proof.  
\end{proof}

Difference equation \eqref{Difference:k} has the trivial solution $x^* = 0$, and when $rT>-\ln(1-E)$ it has a single positive equilibrium
\begin{equation}
\label{pos:equil}
x^* = \frac{((1-E)e^{rT}-1)K_c}{e^{rT}-1} \, .
\end{equation}
If $rT\leq -\ln(1-E)$ then only the trivial equilibrium exists, and as Lemma \ref{theorem:onlyzero} states, all solutions of \eqref{Difference:k}, and hence of \eqref{Impulse:Delay}, will inevitably go to extinction.

\begin{lemma}
\label{theorem:onlyzero}
If
\begin{equation}
\label{zeroonly}
rT\leq -\ln(1-E),
\end{equation}
all solutions of \eqref{Difference:k} tend to zero.
\end{lemma}

\begin{proof}
By 
Lemma~\ref{lemma_bounds}, we can consider $x_n\in [0,K_c]$ for $n$ large enough.
If for some $n \geq k$, $x_n=0$, we have $x_{j} = 0$ for any $j \geq n$ in \eqref{Difference:k}, 
so we restrict ourselves to only considering positive sequences $\{ x_n \}$.
Further, let  $\{ x_n \}$ be an eventually monotone sequence, then it has a limit $d$. If $d=0$,  the sequence converges to zero; if $d>0$, we let $n\to \infty$ in 
$$
 x_{n+1} =  \frac{K_c x_{n}e^{rT}}{K_c + x_{n}(e^{rT}-1)} - E \frac{K_c x_{n-k}e^{rT}}{K_c + x_{n-k}(e^{rT}-1)}
$$
and get that $d$ is a positive equilibrium solution of \eqref{Difference:k}, which is a contradiction.
Thus, we have only to consider sequences $\{ x_n \}$ that are neither eventually non-increasing nor eventually non-decreasing. 

Before we handle this case, let us recall that the function $F(x,a)$ defined in \eqref{F_def}
is strictly increasing in both $x$ and $a$ for $K_c>0$ and $a>1$ (here $a=e^{rT}>1$ for $rT>0$). 

Let $k=1$, since we are only considering sequences that are {\bf not} non-decreasing, this implies that there exists some $n$ such that  $x_n<x_{n-1}$. Then,
\begin{align*}
x_{n+1}  & =  \frac{K_c x_{n}e^{rT}}{K_c + x_{n}(e^{rT}-1)} - E \frac{K_c x_{n-1}e^{rT}}{K_c + x_{n-1}(e^{rT}-1)}
<  \frac{(1-E) K_c x_{n}e^{rT}}{K_c + x_{n}(e^{rT}-1)} \\ 
& \leq \frac{(1-E) K_c x_{n} \frac{1}{1-E}}{K_c + x_{n-k}( \frac{1}{1-E} - 1)} = \frac{ K_c x_{n} }{K_c + x_n \frac{E}{1-E}} < x_n,
\end{align*}
and by induction we get that $\{x_j\}$ is a monotonically decreasing sequence starting with $j=n$, and thus it converges to zero, as justified above.

Next, consider $k\geq 2$.
If there exists $n$ such that $x_n = \min \{ x_{n-k}, x_{n-k+1}, \dots, x_{n-1},x_n \}$, 
then since $F(x_n,e^{rT}) \leq F(x_{n-k},e^{rT})$ and by \eqref{zeroonly}, 
we get
\begin{align*}
x_{n+1}  & =  \frac{K_c x_{n}e^{rT}}{K_c + x_{n}(e^{rT}-1)} - E \frac{K_c x_{n-k}e^{rT}}{K_c + x_{n-k}(e^{rT}-1)}
\leq \frac{(1-E) K_c x_{n}e^{rT}}{K_c + x_{n}(e^{rT}-1)} \\ & \leq \frac{(1-E) K_c x_{n} \frac{1}{1-E}}{K_c + x_{n-k}( \frac{1}{1-E} - 1)} = \frac{ K_c x_{n} }{K_c + x_n \frac{E}{1-E}} < x_n
\end{align*}
as above. Thus $x_{n+1}<x_n$ and $x_{n+1} = \min \{ x_{n-k+1}, x_{n-k+2}, \dots, x_{n},x_{n+1} \}$, 
which yields that $x_{n+2}< x_{n+1}$, and by induction $\{ x_n \}$ is monotonically decreasing and thus converges to zero.

Let us, finally, consider the case when $x_n > \min \{ x_{n-k}, x_{n-k+1}, \dots, x_{n-1} \}$ for any $n \geq n_0 > k$. Then, 
\begin{equation}
\label{liminf}
m := \liminf_{n \to \infty} x_n \geq \min\{x_{n_0 - k},...,x_{n_0}\}.
\end{equation}
Introducing the functions $H$ and $h$ through $F$ as defined in \eqref{F_def} and using inequality \eqref{zeroonly}
leading to $\displaystyle e^{rT} \leq \frac{1}{1-E}$, we obtain
\begin{equation}
\label{def_H}
H(x) := x - (1-E)  F \left( x, e^{rT}, \right)  \geq x - (1-E)  F \left( x, \frac{1}{1-E}  \right)  = 
x - \frac{ K_c x }{K_c + x \frac{E}{1-E}} = : h(x). 
\end{equation}
Note that
$$
h(x) = 
x \left( 1 - \frac{K_c}{K_c + \frac{E}{1-E} x} \right) 
$$
is monotone increasing for $x\in[0,K_c]$ (as a product of two non-negative increasing functions) from $h(0)=0$ to $h(K_c)=E K_c > 0$. Let us choose $\varepsilon>0$ such that 
$h(x)>\varepsilon$ for $x \in [\frac{m}{2},K_c]$. From \eqref{def_H} we get $H (x)>\varepsilon$, $x \in [\frac{m}{2},K_c]$ as well. Define a positive $\displaystyle \delta \in \left( 0, \frac{\varepsilon}{4} \right)$ satisfying 
$$
\delta < \min\left\{ \frac{m}{2}, K_c - \frac{m}{2} \right\}
$$
such that for any $x,y \in [0,K_c]$, the inequality $|x-y| \leq \delta$ implies
$$ \left| F \left( x, e^{rT} \right) - F \left( y, e^{rT} \right) \right| \leq \frac{\varepsilon}{2}. $$
Such $\delta>0$ exists, as $F(x,e^{rT})$ defined in 
\eqref{F_def} is continuous and thus is uniformly continuous for $x\in[0,K_c]$.

By the definition of $m$ in \eqref{liminf} for any $\delta > 0$ there is $n_1 \geq n_0+k$ such that $x_n > m - \frac{\delta}{2}$ for any $n \geq n_1 - k$. By definition of $\liminf$, there is $n>n_1$ such that $x_n < m + \frac{\delta}{2}$. Also, by the choice of $n_1$, we have $x_{n-k} > m- \frac{\delta}{2}$ and thus $|x_n - x_{n-k}|<\delta$.
Further,
\begin{align*}
x_{n+1}  & =  F \left( x_n, e^{rT} \right) - E F \left( x_{n-k}, e^{rT} \right) 
\\ & 
=  F \left( x_n, e^{rT}, \right) - E F \left( x_n, e^{rT} \right) 
+ E \bigg[ F \left( x_n, e^{rT}  \right) -  F \left( x_{n-k}, e^{rT}  \right) \bigg] \\
& = x_n - H(x_n) + E \bigg[ F \left( x_n, e^{rT}  \right) -  F \left( x_{n-k}, e^{rT}   \right) \bigg] 
\\ & 
\leq x_n -h(x_n) + E \bigg| F \left( e^{rT},x_n \right) -  F \left( x_{n-k}, e^{rT}  \right) \bigg|
\\
& < x_n - \varepsilon + E \frac{\varepsilon}{2} < x_n - \varepsilon +  \frac{\varepsilon}{2} 
 = x_n -  \frac{\varepsilon}{2} < x_n - 2 \delta \leq m +\delta - 2 \delta = m -\delta < m- \frac{\delta}{2},
\end{align*}
which contradicts to the assumption that $x_n> m-\delta/2$ for any $n \geq n_1-k$. 
Thus the scenario described by \eqref{liminf} is impossible, and the solution either coincides with zero, 
starting at some $x_j$, or tends to zero.
\end{proof}

\begin{theorem}
\label{contin:zeroonly}
If \eqref{zeroonly} holds, all solutions of \eqref{Impulse:Delay} tend to zero.
\end{theorem}

\begin{proof}
By Lemma~\ref{theorem:onlyzero}, if \eqref{zeroonly} holds then all solutions of the difference equation tend to zero. This implies that $N(nT^+)\to 0$, and by \eqref{Impulse:Sol} the solution of \eqref{Impulse:Delay} $N(t)$ on $(nT,(n+1)T)$ also tends to zero. Therefore all solutions of \eqref{Impulse:Delay} tend to zero.
\end{proof}


By 
Theorem~\ref{contin:zeroonly}, it is sufficient to consider only the case 
\begin{equation}
\label{3:LessHalfStable}
E< 1 - e^{-rT}.
\end{equation}
If this condition is not satisfied, all solutions tend to zero. Everywhere below we assume that \eqref{3:LessHalfStable} holds.

Next, let us focus on the behaviour of the difference equation when the harvesting is delayed by a single time period $k=1$. This will lead to the second-order difference equation, and allow us to apply necessary and sufficient results such as Lemma~\ref{Cor:CritK1} and \ref{Lemma:OnUnitDisk} to obtain explicit conditions for stability of the difference equation.

\begin{lemma}
\label{Stability:k1}
Let $k=1$. 
If $E \in (0,1/2]$ then there exists a positive equilibrium of \eqref{Difference:k} which is locally asymptotically stable.

For $E\in (1/2,1)$, if
\begin{equation}
\label{3:MoreHalfUnstable}
rT < -\ln\bigg(\frac{(1-E)^2}{E} \bigg)
\end{equation}
then the positive equilibrium of \eqref{Difference:k} is locally unstable, while if
\begin{equation}
\label{3:MoreHalfStable}
rT > - \ln\bigg( \frac{(1-E)^2}{E}\bigg),
\end{equation}
the positive equilibrium of \eqref{Difference:k} is locally asymptotically stable.
\end{lemma}

\begin{proof}
As \eqref{3:LessHalfStable} holds, 
there exists a unique positive equilibrium $x^*$. 

Let $k=1$, then \eqref{Difference:k} has the form $x_{n+1}=\max\{ f(x_n,x_{n-1}), 0\}$, and the linearized equation around $x^*$ is
\begin{equation*}
    u_{n+1} = p_0 u_{n} - p_1 u_{n-1},
\end{equation*}
where
\begin{equation*}
\begin{aligned}
    p_0 &= \frac{\partial f}{\partial x_{n}}(x^*,x^*) = \frac{K^2 e^{rT}}{(K + x_n(e^{rT}-1))^2} \bigg|_{(x^*,x^*)} = \frac{e^{-rT}}{(1-E)^2} \, ,\\
    p_{1} &=  - \frac{\partial f}{\partial x_{n-1}}(x^*,x^*) = E \frac{K^2 e^{rT}}{(K + x_{n-1}(e^{rT}-1))^2} \bigg|_{(x^*,x^*)} = \frac{E e^{-rT}}{(1-E)^2}.
\end{aligned}
\end{equation*}
The characteristic equation of the linearized equation is $\lambda^2 - p_0 \lambda + p_1 = 0$.
By Lemma~\ref{Cor:CritK1}, for the roots of the characteristic equation to lie inside the unit disc, we only require $p_1<1$ (the left inequality $p_0 - 1< p_1$ is automatically satisfied for all $rT>-\ln(1-E)$), the inequality has the form
\begin{equation*}
\label{K1:Conditions}
rT > -\ln\bigg(\frac{(1-E)^2}{E}\bigg),
\end{equation*}
which coincides with \eqref{3:MoreHalfStable}.

If $E\in (0,1/2]$ then $-\ln((1-E)^2/E) \leq -\ln(1-E)$, and 
therefore by \eqref{3:LessHalfStable}, the positive equilibrium $x^*$ is locally asymptotically stable, as \eqref{3:LessHalfStable} implies \eqref{3:MoreHalfStable}.

If $E\in ( 1/2,1)$, 
\eqref{3:MoreHalfStable} implies that the positive equilibrium $x^*$ exists and is locally asymptotically stable.

When $E \in (1/2,1)$,  if $-\ln(1-E) < rT \leq  - \ln( (1-E)^2/E)$, the positive equilibrium exists and the roots of the characteristic equation satisfy $\max\{|\lambda_1|,|\lambda_2|\}\geq 1$. By Lemma~\ref{Lemma:OnUnitDisk},  a root of the characteristic equation satisfies $|\lambda| = 1$ if and only if $rT = -\ln((1-E)^2/E)$ on this interval. 
Really, in this case $p_1=1$ and $0<p_0= \frac{1}{E}<2$.
Therefore for $-\ln(1-E)< rT < -\ln((1-E)^2/E)$, which corresponds to \eqref{3:MoreHalfUnstable}, we get $\max\{|\lambda_1|,|\lambda_2|\} > 1$ and by linearization the equilibrium is unstable.

\end{proof}

Lemmata~\ref{Stability:k1} and \ref{connection} immediately imply

\begin{theorem}
\label{Thm:PerSolK1}
Let k=1. If $rT > - \ln\bigg( \frac{(1-E)^2}{E}\bigg)$, 
then there exists a unique positive periodic solution $N^*(t)$ of \eqref{Impulse:Delay} with 
\begin{equation}
\label{NSol}
N^* (nT^+) = x^* = \frac{((1-E)e^{rT}-1)K_c}{e^{rT}-1} \,,
\end{equation}
and this solution is locally asymptotically stable.
\end{theorem}

While Theorem~\ref{Thm:PerSolK1} is similar in many ways to the result of \cite{Zhang2003} 
cited in Lemma~\ref{ZhangThm}, here we do not observe attractivity of the solution for all initial conditions, implying that the solution is attractive but not globally attractive. Even if the equilibrium is locally asymptotically stable, there are sets of initial values leading to its instability, as further numerical examples illustrate.

\begin{remark}
\label{rem:extinction}
Even for $k=1$ and any $x_{0}$, there is a domain of initial values $x_1$ guaranteeing immediate extinction. Really, as 
$\displaystyle F(x,e^{rT})$ defined in \eqref{F_def}
is strictly increasing in $x$, $F(0,e^{rT})=0$, there are values of $x_1<x_{0}$ such that
$$
F \left( x_1,e^{rT} \right) = \frac{K_c x_1 e^{rT}}{K_c+x_1 (e^{rT}-1)} \leq \frac{E K_c x_{0} e^{rT}}{K_c+x_{0}(e^{rT}-1)}
= E F\left( x_{0} ,e^{rT} \right),
$$
leading to $x_2=x_3 = \dots = 0$.
\end{remark}

The following results extends Lemma~\ref{Stability:k1} to all $k\in {\mathbb N}$, however,
for a given $k$, we have implicit, but easily verifiable conditions connecting $rT$ with $E$. 

\begin{lemma}
\label{Cor:AnyE}
The positive equilibrium $x^*$ of difference equation \eqref{Difference:k} is asymptotically stable
if and only if both inequalities hold
\begin{equation}
\label{spec_stab_cond}
\frac{e^{-rT}}{(1-E)^2} < \frac{k+1}{k}, \quad \cos(\theta^*)< \frac{e^{-rT}(1+E)}{2 (1-E)} + \frac{(1-E)^2 e^{rT}}{2} \, , 
\end{equation}
where $\theta^*$ is a solution of 
\begin{equation}
\label{theta_1}
\frac{\sin(k\theta)}{\sin((k+1)\theta)} = (1-E)^2 e^{rT}, \quad \theta \in \left( 0, \frac{\pi}{k+1} \right).
\end{equation}
\end{lemma}

\begin{proof}
Linearizing the difference equation $x_{n+1} = \max\{f(x_n, x_{n-k}),0\}$ around the positive equilibrium $x^*$ given in \eqref{pos:equil} with
 \begin{equation*}
 f(x_n, x_{n-k}) = \frac{K_c x_{n}e^{rT}}{K_c + x_{n}(e^{rT}-1)} - E \frac{K_c x_{n-k}e^{rT}}{K_c + x_{n-k}(e^{rT}-1)},
 \end{equation*}
we get, similarly to the proof of Lemma~\ref{Stability:k1},
\begin{equation}
\label{flower1}
    u_{n+1} = p_0 u_{n} - p_k u_{n-k}, \quad p_0 = \frac{e^{-rT}}{(1-E)^2}, \quad  p_k = \frac{E e^{-rT}}{(1-E)^2}.
\end{equation}
%
By Corollary~\ref{Cor:MoreGeneralk}, the zero solution of the linearized equation is asymptotically stable if and only if the inequalities in \eqref{gen_stab_cond} hold, or equivalently,
$$
\frac{e^{-rT}}{(1-E)^2}  < \frac{k+1}{k}, \quad \frac{E e^{-rT}}{(1-E)^2} < \sqrt{ \frac{e^{-2rT}}{(1-E)^4} + 1 - 2
\frac{e^{-rT}}{(1-E)^2} \cos(\theta^*)} \, .
$$
Note that $p_0< p_k+1$, or $\displaystyle e^{-rT} < E e^{-rT} + (1-E)^2 $ is equivalent to $\displaystyle e^{-rT}< 1-E$, which is satisfied due to \eqref{3:LessHalfStable}.
The first inequality is the same as in \eqref{spec_stab_cond}, while computing the squares in the second one gives
\begin{equation*}
\frac{E^2 e^{-2rT}}{(1-E)^4} < \frac{e^{-2rT}}{(1-E)^4} + 1 - \frac{2 e^{-rT}}{(1-E)^2}\cos(\theta^*).
\end{equation*}
After rearranging, the desired result is acquired.
\end{proof}

Applied to \eqref{Impulse:Delay}, Lemmata~\ref{Cor:AnyE} and \ref{connection}  immediately imply a sharp local asymptotic stability result.

\begin{theorem}
\label{th:AnyE}
Let $E\in(0,1)$ and \eqref{3:LessHalfStable} be satisfied.
Then there exists a unique positive periodic solution $N^*(t)$ of \eqref{Impulse:Delay} given by \eqref{NSol}
which is asymptotically stable if and only if both inequalities in \eqref{spec_stab_cond} hold, where $\theta^*$ is a solution of 
\eqref{theta_1}.
\end{theorem}

The following stability condition is delay-independent; moreover, it  applies if a constant $k$ in \eqref{Difference:k} is replaced with $k=k(n)$. It is also only sufficient, meaning there may still exist values of $rT< \ln(E + 1) - 2 \ln(1-E)$ such that the equilibrium is locally asymptotically stable.

\begin{theorem}
\label{th:contraction}
If
\begin{equation}
\label{Cond:k}
    rT > \ln(E + 1) - 2 \ln(1-E)
\end{equation}
then the positive periodic solution $N^*(t)$ of \eqref{Impulse:Delay} exists and corresponds to \eqref{NSol}, and this solution is locally asymptotically stable.
\end{theorem}
\begin{proof}
If \eqref{Cond:k} holds, then so does \eqref{3:LessHalfStable}, and the positive equilibrium exists. 
Reducing \eqref{Impulse:Delay} to difference equation \eqref{Difference:k}, we once again linearize \eqref{Difference:k} around $x^*$ and obtain \eqref{flower1}, where by \eqref{Cond:k},
$$
|p_0| + |p_k| = \frac{1+E}{(1-E)^2} e^{-rT} <1.
$$

By Lemma~\ref{sufficient}, the equilibrium $x^*$, and by Lemma~\ref{connection} the positive periodic solution $N^*(t)$ of \eqref{Impulse:Delay} is locally asymptotically stable, once \eqref{Cond:k} holds.
\end{proof}

\section{MSY with Delay Impulsive Harvesting}
\label{sec:MSY}

Next, we proceed with the analysis of a maximum yield (MY) and a maximum sustainable yield (MSY). We recall that a yield is said to be sustainable if it corresponds to a solution that is at least locally asymptotically stable.

\begin{lemma}
\label{Lem:yield}
The yield of \eqref{Impulse:Delay} is associated to the solution $N^*(nT^+) = x^*$ of \eqref{Difference:k} and is given by 
\begin{equation}
\label{yield}
Y(E) = \frac{K_c E}{(1-E)T}\bigg(\frac{(1-E)e^{rT}-1}{e^{rT}-1} \bigg).
\end{equation}
This yield is a function of the harvesting effort $E$, it is an increasing function for $E\in(0,E_{opt})$ and is decreasing 
for $E\in(E_{opt},1)$.
\end{lemma}

\begin{proof}
For an optimal $T$-periodic solution of \eqref{Impulse:Delay}, we get $N^*(nT^+) = N^*((n+1)T^+) = x^*$, where \eqref{NSol} holds.
Then the associated yield is 
\begin{equation*}
Y(E) = \frac{E N^*(nT)}{T} = \frac{E}{1-E}\cdot \frac{N^*(nT^+)}{T} = \frac{E}{1-E}\cdot ~ \frac{K_c}{T}\bigg(\frac{(1-E)e^{rT}-1}{e^{rT}-1} \bigg) .
\end{equation*}
Its derivative in $E$,
$$
Y'(E) = \frac{K_c}{T (e^{rT}-1)} \frac{d}{dE} \bigg[ E e^{rT} + 1 - \frac{1}{1-E} \bigg]
= \frac{K_c}{T (e^{rT}-1)} \bigg[  e^{rT} -  \frac{1}{(1-E)^2} \bigg].
$$
satisfies $Y'(E)>0$ for $(1-E)^2>e^{-rT}$, which is equivalent to $E\in(0,E_{opt})$, and $Y'(E)<0$ for $E\in(E_{opt},1)$.
\end{proof}

\begin{lemma}
\label{Thm:MaxYield}
The maximum yield (MY) for delayed impulsive harvesting model \eqref{Impulse:Delay} with $k\in {\mathbb N}$ is equal to the MY for non-delayed model \eqref{Impulse:NoDelay} with  
the optimal harvesting effort  $E_{opt} = 1 - e^{-rT/2}$ and  MY associated to  
\eqref{eq:optsol}. 
\end{lemma}

\begin{proof}
By Lemma \ref{Lem:yield}, MY is attained at $E = E_{opt}$ and has the value of yield per time
\begin{equation*}
MY_{delayed} = Y(E_{opt}) = \frac{K_c}{T}\bigg( \frac{e^{rT/2}-1}{e^{rT/2}+1}\bigg) = MY_{non-delayed}.
\end{equation*}
In addition, the periodic solution when $E = E_{opt}$ becomes $\displaystyle N^*(nT^+) = \frac{K_c}{e^{rT/2}+1}$,
as in \eqref{eq:optsol}.
\end{proof}

\begin{theorem}
\label{Thm:Stable}
The solution corresponding to the MY of \eqref{Impulse:Delay} 
satisfies \eqref{eq:optsol} with $E = E_{opt}$. 
MY is a MSY if and only if either $k=1$ or $k \geq 2$ and
\begin{equation}
\label{Cond:General}
 rT < - 2 \ln\bigg(1 - 2 \cos\bigg(\frac{k\pi}{2k + 1}\bigg) \bigg)\,.
\end{equation}
\end{theorem}

\begin{proof}
By Lemma \ref{Lem:DiffSol} the solutions to \eqref{Impulse:Delay} satisfy \eqref{Difference:k}.
Linearizing the difference equation $x_{n+1} = f(x_n, x_{n-k})$ around the positive equilibrium $x^*$ given in \eqref{pos:equil}, we get \eqref{flower1}.

By Lemma~\ref{Thm:MaxYield}, the yield is maximal whenever $E = E_{opt} = 1 - e^{-rT/2}$, and the linearized difference equation becomes
\begin{equation*}
u_{n+1} = u_n - (1 - e^{-rT/2})u_{n-k}
\end{equation*}
with equilibrium \eqref{eq:optsol}. Then by Lemma~\ref{Lem:Generalk}, for $k=1$, as $1-e^{-rT/2} \in (0,1)$, the solution is asymptotically stable for any $E \in (0,1)$. 
For $k>1$,
the zero solution of the linearized equation is locally asymptotically stable if and only if
\begin{equation*}
1 - e^{-rT/2} < 2 \cos\bigg(\frac{k \pi}{2k + 1} \bigg).
\end{equation*}
The condition is equivalent to \eqref{Cond:General}, leading to local asymptotic stability for the positive equilibrium $x^* = K_c/(e^{rT/2}+1)$ of  \eqref{Difference:k}.
Finally, by Lemma~\ref{connection}, a solution of \eqref{Impulse:Delay} satisfying \eqref{eq:optsol} is locally asymptotically stable, once \eqref{Cond:General} is satisfied. By definition, a unique positive periodic solution $N^*(t)$  with \eqref{eq:optsol}, for either $k=1$ or both $k\geq 2$ and $rT$ satisfying \eqref{Cond:General}, leads to MSY.
\end{proof}

Unlike the non-delay case, there is a possibility that the maximum yield is not sustainable. To avoid extinction, the choice of harvesting efforts should be among those leading to a sustainable yield. The set of such efforts is non-empty, as the following statement guarantees.

\begin{theorem}
\label{Thm:MSYG}
Let $k \geq 2$ and
\begin{equation}
\label{E_star}
E^* = \frac{2+e^{-rT} - \sqrt{e^{-rT}(e^{-rT}+8)}}{2} \, .
\end{equation}
Then $E^*<E_{opt}$, and for any $E\in(0,E^*)$ the yield as given in \eqref{yield} is sustainable.
\end{theorem}

\begin{proof}
Let us note that, first, $E^*$ defined in \eqref{E_star} is positive and, as
$4 e^{-rT/2} > 4 e^{-rT}$, we have $$\sqrt{e^{-rT}(e^{-rT}+8)} >  e^{-rT} + 2 e^{-rT/2}$$
leading to $E^*< E_{opt} = 1-e^{-rT/2}$.

For a fixed $E \in (0,1)$, we get a solution with $N^*(nT^+) = x^*$, corresponding to a yield as given in \eqref{yield}. As justified earlier, $Y(E)$ is an increasing function of $E$ for $E\in(0,E_{opt})$.

By Theorem \ref{th:contraction} and Lemma \ref{connection}, the solution is locally asymptotically stable for any $k$ if $\displaystyle e^{rT}>\frac{E+1}{(1-E)^2}$, which is equivalent to
$E^2 - (2+e^{-rT})E + 1 - e^{-rT}>0$. The quadratic inequality is also satisfied if $E \in (0,E^*)$, meaning that for any $E\in(0,E^*)$ the yield is sustainable, which concludes the proof.
\end{proof}

\begin{lemma}
\label{Thm:LessSus}
If for some choice of $E_s \in(0,E_{opt}]$ the associated yield is sustainable, then for any $E\in(0,E_s]$ the yield associated with $E$ is also sustainable.
\end{lemma}

\begin{proof}
If the yield associated with $E_s$ is sustainable, the associated solution with $N^*(nT^+)=x^*$ is locally asymptotically stable.
By Lemma \ref{Cor:AnyE}, both inequalities in \eqref{spec_stab_cond} must hold for $E_s$, where $\theta^*$ is a root of \eqref{theta_1}.

Since for any $E\leq E_{opt}$, 
\begin{equation*}
\frac{e^{-rT}}{(1-E)^2}\leq \frac{e^{-rT}}{(1-E_{opt})^2} = 1 < \frac{k+1}{k},
\end{equation*}
it is clear that the first inequality in \eqref{spec_stab_cond} is satisfied for any $E\in(0,E_{opt}]$. Thus we can turn our attention to the second inequality, denote its right-hand side for a fixed $rT$ as $h_1$,
\begin{equation}
\label{def_h1}
h_1(E) := \frac{(1+E)e^{-rT}}{2(1-E)} + \frac{(1-E)^2e^{rT}}{2} \, ,
\end{equation}
and the left-hand side in \eqref{theta_1} as
\begin{equation}
\label{def_h2}
h_2(\theta(E)) = \frac{\sin(k \theta(E))}{\sin ((k+1)\theta(E))}, \quad \theta \in I = (0,\pi/(k+1)).
\end{equation}
We have from \eqref{theta_1},
$$
h_2^{\prime} (\theta(E)) \frac{d \theta}{dE} = \frac{d \theta}{dE} \left[ (1-E)^2 e^{rT} \right] = -2(1-E) e^{rT} < 0.
$$
Also,
\begin{equation*}
\begin{aligned}
h_2^{\prime}(\theta) &= \frac{k \cos(k \theta) \sin((k+1)\theta) - (k+1)\cos((k+1)\theta)\sin(k\theta)}{\sin^2((k+1)\theta)}\\
&= k \frac{\sin((k+1)\theta)\cos(k\theta) - \sin(k\theta)\cos((k+1)\theta)}{\sin^2 ((k+1)\theta) }  - \frac{\cos((k+1)\theta)\sin(k\theta)}{\sin^2((k+1)\theta)}\\
&= k\frac{\sin(\theta)}{\sin^2 ((k+1)\theta) } - \frac{\cos((k+1)\theta)\sin(k\theta)}{\sin^2 ((k+1)\theta) } \\
&= \frac{\sin(\theta)}{\sin^2 ((k+1)\theta) }\bigg( k - \frac{\cos((k+1)\theta)\sin(k \theta)}{\sin(\theta)} \bigg)
\end{aligned}
\end{equation*}
Now, $\sin(\theta)/\sin((k+1)\theta)^2 > 0$ $\forall \theta \in I$, and we will show that $k\sin(\theta) - \cos((k+1)\theta)\sin(k \theta) > 0$ leading to $h_2^{\prime}(\theta) > 0$.
Since $\cos((k+1)\theta)\leq 1$,
 $$k \sin(\theta)- \cos((k+1)\theta)\sin(k \theta) > k \sin(\theta) - \sin(k \theta) =: H_1(\theta),$$
where $H_1(0) = 0$ and $H_1^{\prime}(\theta) = k(\cos(\theta) - \cos(k\theta))>0$ since $\cos(\theta)$ is decreasing for all $\theta \in (0,\pi)$ and for $\theta\in I$, both $\theta$ and $k \theta$ are in $(0,\pi)$. Thus $H_1(\theta)>0$ for $\theta \in I$.
 
 Since $h_2^{\prime}(\theta) > 0$, the inequality $\displaystyle h_2^{\prime} (\theta(E)) \frac{d \theta}{dE} <0$ implies 
 $\displaystyle \frac{d \theta}{d E} < 0$. 
 Thus $\theta(E)$ decreases, and  $\cos(\theta(E))$ increases in $E$ as well. Further,
 $$
 h^{\prime}_1(E) = \frac{e^{-rT}}{(1-E)^2} - (1-E)e^{rT} < 0
 $$
 for $E < 1 - e^{-2 rT/3}$. By our assumption, $E\leq E_{opt} = 1 - e^{- rT/2} < 1 - e^{-2 rT/3}$, therefore $h^{\prime}_1(E)<0$, and $h_1(E)$ decreases in E.
 
 Since the yield associated with $E_s$ is sustainable, we have
 $$
 \cos(\theta(E_s)) < h_1 (E_s).
 $$
 Since $\cos(\theta(E))$ decreases, and $h_1(E)$ increases for decreasing $E$, then for any $E\leq E_s$
 $$
 \cos(\theta(E)) < h_1(E),
 $$
 and the second inequality in \eqref{spec_stab_cond} is satisfied.
 
Since both inequalities in \eqref{spec_stab_cond} are satisfied for $E\leq E_s$, the solution associated with $E$ with $N^*(nT^+) = x^*$ is locally asymptotically stable, and the yield associated with $E$ is sustainable.
\end{proof}

\begin{corollary}
Let $k=1$, then for any $E\in(0,E_{opt}]$ the yield is sustainable.
\end{corollary}

The statement follows immediately from Lemma~\ref{Thm:LessSus} and Theorem~\ref{Thm:Stable},
while Lemma~\ref{Thm:LessSus} and Theorem~\ref{Thm:Stable} imply

\begin{corollary}
Let $k\geq 2$, and $rT < - 2 \ln\bigg(1 - 2 \cos\bigg(\frac{k\pi}{2k + 1}\bigg) \bigg)$. Then for any $E\in(0,E_{opt}]$ the yield is sustainable.
\end{corollary}

Theorem~\ref{Thm:MSYG_opt} determines the maximum bound for a sustainable yield.

\begin{theorem}
\label{Thm:MSYG_opt}
Let $k \geq 2$, and $rT \geq - 2 \ln\bigg(1 - 2 \cos\bigg(\frac{k\pi}{2k + 1}\bigg) \bigg)$. 
Then there exist $E^{**}\in(0,E_{opt}]$ and $\theta^* \in \left( 0, \frac{\pi}{k+1} \right)$ such that $(E^{**},\theta^*)$ is a unique solution of 
\begin{equation}
\label{add_remark}
\cos(\theta) = \frac{(1+E) e^{-rT}}{2 (1-E)}+ \frac{(1-E)^2 e^{rT}}{2}, 
\quad
\frac{\sin(k\theta)}{\sin((k+1)\theta)} = (1-E)^2 e^{rT}, \quad \theta \in \left( 0, \frac{\pi}{k+1} \right).
\end{equation}
For any $E\in(0,E^{**})$, the yield is sustainable, while for $E\in [E^{**},E_{opt}]$ the yield is not sustainable.
\end{theorem}

\begin{proof}
Let us prove that $E^{**}$ exists and is unique. By Theorem~\ref{Thm:Stable}, 
when $k\geq 2$, $rT \geq - 2 \ln\bigg(1 - 2 \cos\bigg(\frac{k\pi}{2k + 1}\bigg) \bigg)$ and $E=E_{opt}$, the associated solution is unstable and $\cos(\theta(E_{opt})) \geq h_1(E_{opt})$. On the other hand, when $E\to 0^+$ 
we get in \eqref{spec_stab_cond}, $\displaystyle \cos(\theta(0^+)) = \frac{e^{-rT}}{2}+ \frac{e^{rT}}{2}<\cosh(rT) = h_1(0^+)$. 

As $h_2(\theta(E))$ defined in \eqref{def_h2} is continuous and monotone decreasing in $E$, $E(\theta)$ is continuous and monotone for $\theta \in I = (0,\pi/(k+1))$, and even on $[0,\pi/(k+1))$, if we define $h_1(0)=k/(k+1)$. The inverse function is also monotone and continuous for $E \in [0,E_{opt}]$, as well as $\cos(\theta(E))$.
When $E \to 0^+$, $\cos(\theta(0^+)) < h_1(0^+)$ and by assumption at the optimal harvesting effort $\cos(\theta(E_{opt})) \geq h_1(E_{opt})$. By the continuity of $\cos(\theta(E))$ and the Intermediate Value Theorem, there exists a solution $(\theta(E^{**}),E^{**})$ of \eqref{add_remark} with $E^{**}\in(0,E_{opt}]$. Moreover, by monotonicity this value is unique. 

Since by Lemma \ref{Lem:yield} the yield $Y(E)$ is increasing for $E\in(0,E_{opt}]$, the value of $Y(E^{**})$ is an upper bound, it does not satisfy \eqref{spec_stab_cond} and thus is not sustainable.

However for $E = (E^{**})^- = \lim_{E\to (E^{**})^-} E$, by the argument that as $E$ decreases, $\cos(\theta(E))$ decreases and $h_1(E)$ increases, $(E^{**})^-$ will satisfy the conditions in \eqref{spec_stab_cond}, and the associated yield is sustainable. Then by Lemma~\ref{Thm:LessSus}, for any $E\in(0,E^{**})$ the associated yield is sustainable.
\end{proof}

\section{Numerical Simulations and Discussion}
\label{sec:concl}

Let us illustrate sustainability of the optimal yield with simulations.


\begin{figure}[ht]
\label{Fig:SvsU}
\centering
\includegraphics[width=0.48\linewidth]{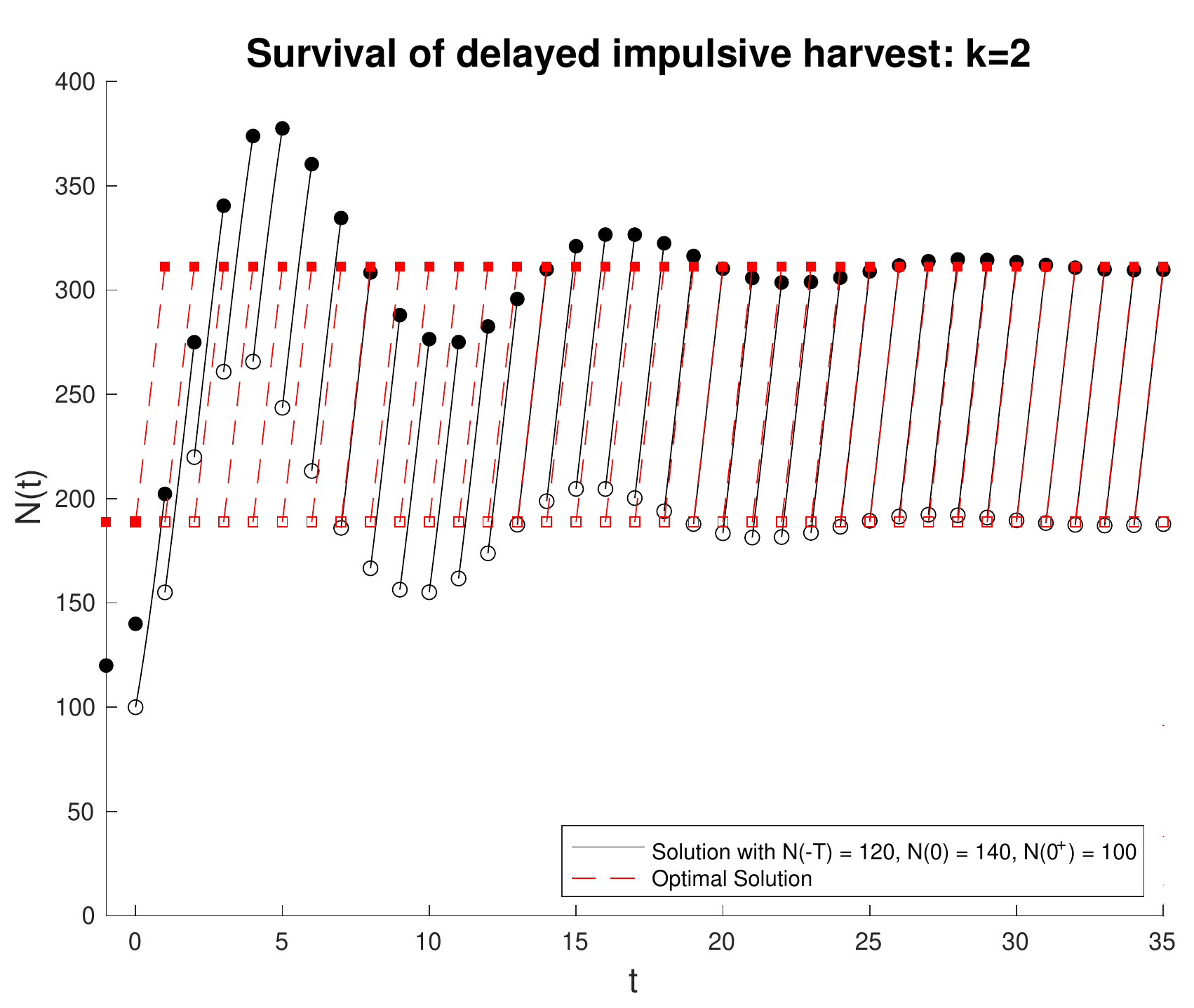}~
\includegraphics[width=0.48\linewidth]{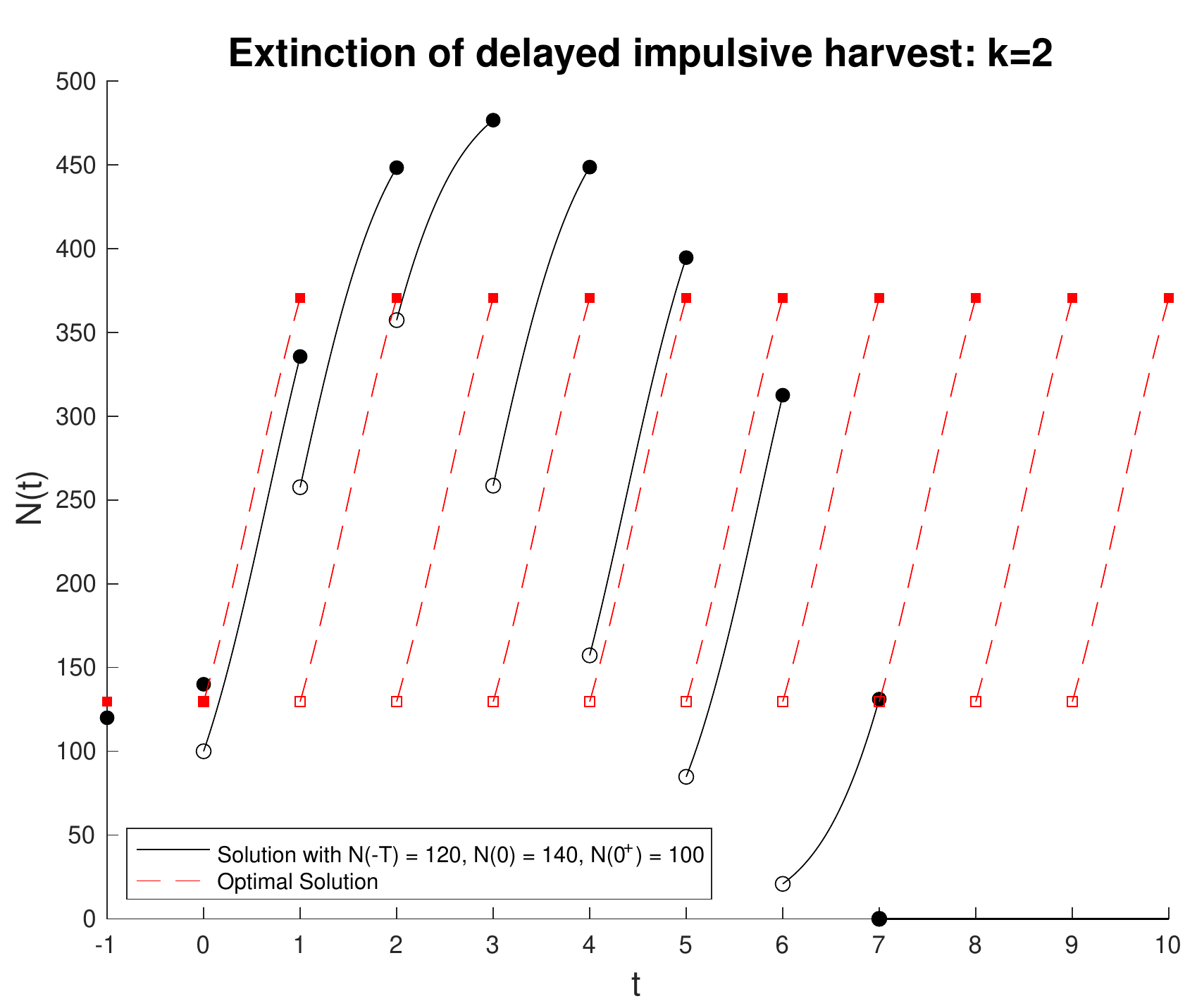}
\caption{Solutions to \eqref{Impulse:Delay} with $k=2$, $K_c = 500$, and $E = E_{opt} = 1 - e^{-rT/2}$. The figure on the left shows the optimal solution to \eqref{Impulse:Delay} (red), and the solution to \eqref{Impulse:Delay} with $r=1$, $T=1$ and initial conditions $N(-T) = 120$, $N(0) = 140$, $N(0^+) = 100$ (black). Since $0< rT < 1.9248$ the optimal solution is locally asymptotically stable.  The figure on the right shows the optimal solution to \eqref{Impulse:Delay} (red), and the solution to \eqref{Impulse:Delay} with $r=2.1$, $T=1$ and initial conditions $N(-T) = 120$, $N(0) = 140$, $N(0^+) = 100$ (black). Since $rT > 1.9248$ the optimal solution is unstable.}
\label{figure1}
\end{figure}

In Fig.~\ref{figure1}, stable and unstable solutions of \eqref{Impulse:Delay} corresponding to the optimal yield are compared. We assume that the delay is two impulsive periods corresponding to $k=2$. To investigate stability, solutions were computed until one of the following became true: the relative error of $|N((n+1)T^+)-N(nT^+)|/N(nT^+)$ was consistently less than $10^{-4}$, $t=100T$, or the population went to extinction, i.e. $N(nT^+) = 0$. The carrying capacity was chosen as $K_c = 500$, and the harvesting effort was chosen to be $E = E_{opt} = 1 - e^{-rT/2}$. Since $k=2>1$, Theorem \ref{Thm:Stable} states that the optimal solution is locally asymptotically stable if and only if $0<rT<1.9248$. In both cases the optimal solution is plotted with a dashed red line, while a solution with initial conditions which are different than the optimal solution is shown via a black solid line.

In Fig.~\ref{figure1}, left, $r=1$, $T=1$ and the initial conditions are $N(-T) = 120$, $N(0) = 140$, $N(0^+) = 100$.  Since $0< rT = 1 < 1.9248$, by Theorem \ref{Thm:Stable} the optimal solution is locally asymptotically stable. Thus we observe a quick convergence to the optimal solution, indicating that the population survives. In Fig.~\ref{figure1}, right, $r=2.1$, $T=1$ and the initial conditions are $N(-T) = 120$, $N(0) = 140$, $N(0^+) = 100$. As $rT = 2.1 > 1.9248$, by Theorem \ref{Thm:Stable} the optimal solution is unstable. Although our solution begins close to the optimal solution due to the imposed initial conditions, we observe the population oscillating slightly before going to extinction at $t=7$.



\begin{figure}[ht]
\centering
\label{Fig:OptDots}
\includegraphics[width=0.5\linewidth]{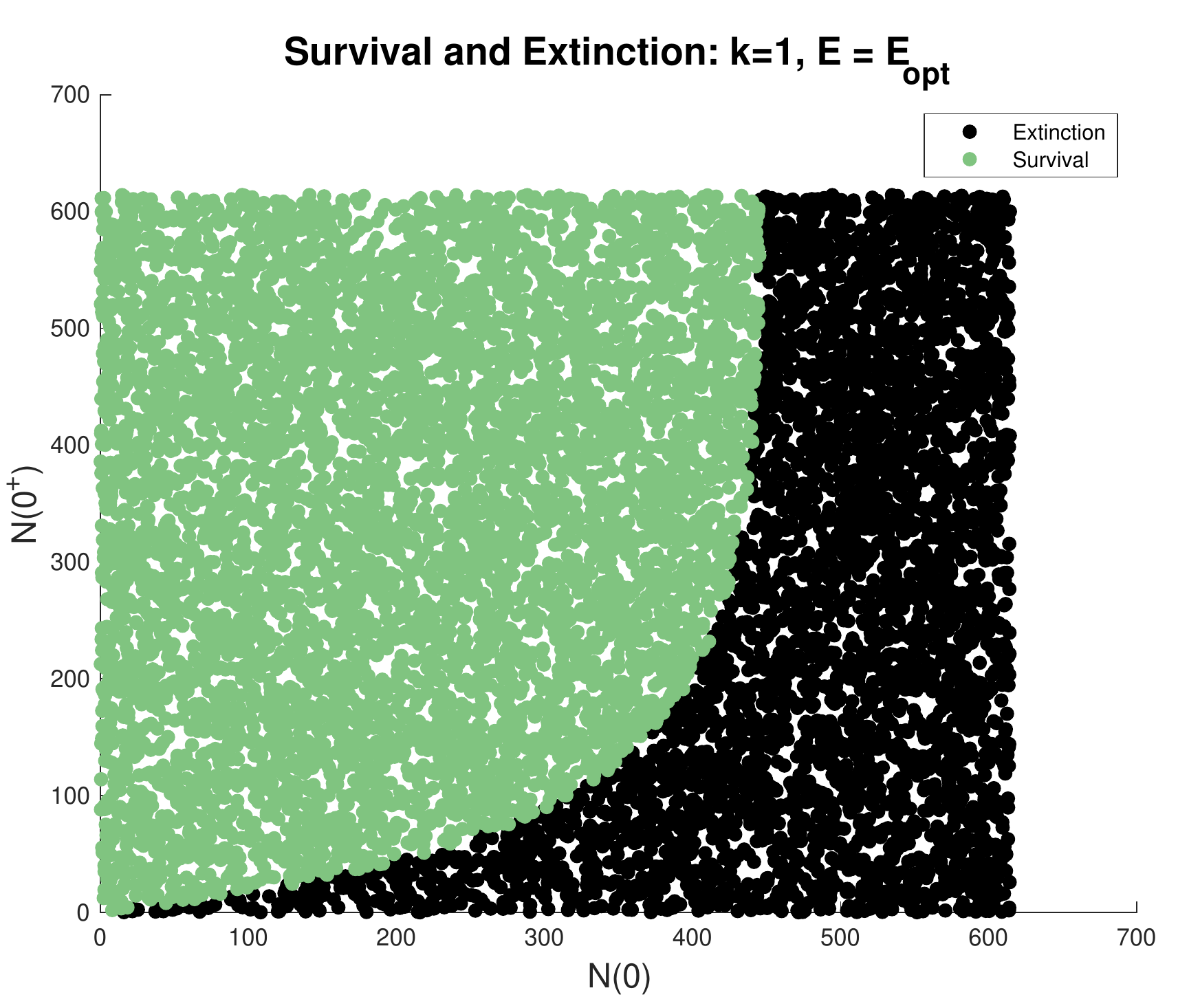}
\caption{When $k=1$, $r=1.3747$, $T=1$, $K_c = 307.1609$, $E = 1 - e^{-(1.3747)(1)/2} = 0.4971$ the positive equilibrium solution with 
$N^*(nT^+) \approx 102.7816$ is locally asymptotically stable. The dots in this figure show whether or not the population survives or goes extinct with the given initial conditions.}
\label{figure2}
\end{figure}

One thing that we wish to highlight is that local stability of the positive periodic solution does not guarantee population survival for all possible positive initial conditions, as Fig.~\ref{figure2} illustrates. In this figure, each dot represents a solution to \eqref{Impulse:Delay} ($k=1$) with a set of initial values $N(0)$ and $N(0^+)$.  The parameters are $k=1$, $r=1.3747$, $T=1$, $K_c = 307.1609$, $E = E_{opt} = 1 - e^{-(1.3747)(1)/2} = 0.4971$ and thus by Theorem \ref{Thm:Stable} (see also Theorem \ref{Thm:PerSolK1}) the positive periodic solution with 
$N^*(nT^+) \approx 102.7816$ is locally asymptotically stable. Fig.~\ref{figure2} tested the global stability of the positive periodic solution by investigating whether or not different combinations of initial conditions would result in survival of the population.
For each solution, both $N(0)$ and $N(0^+)$ were chosen from a uniform distribution of numbers in $[0,2 K_c] = [0,614.3218]$. 
If after $10 500$ iterations the solution $N(nT^+)$ stayed positive (in fact, within $[N^*(nT^+)-10,N^*(nT^+)+10]$), then we say that the population has survived and the population is depicted by a green dot on Fig.~\ref{figure2}. If at any point within the $10 500$ iterations the size of the population after harvesting was less than $10^{-3}$ then the population was said to have gone to extinction, and the population was depicted by a black dot. 

Since the optimal positive periodic solution with $N^*(nT^+) \approx 102.7816$ is locally asymptotically stable, it is natural that we observe the population surviving for wide ranges of initial values. However, if $N(0)$ is much larger than $N(0^+)$ or if $N(0)$ is just generally very large, the population does not survive and goes to extinction. This highlights the lack of global attractivity of the optimal solution which is locally asymptotically stable for all $rT>0$.

We proved that asymptotic stability of the difference equation which corresponds to the optimal solution of impulsive model \eqref{Impulse:Delay} is $k$-dependent. In fact, we can show that as $k\to \infty$, the range of possible $rT$ values that allow a locally asymptotically stable equilibrium shrinks.

\begin{figure}[ht]
\centering
\includegraphics[width=0.5\linewidth]{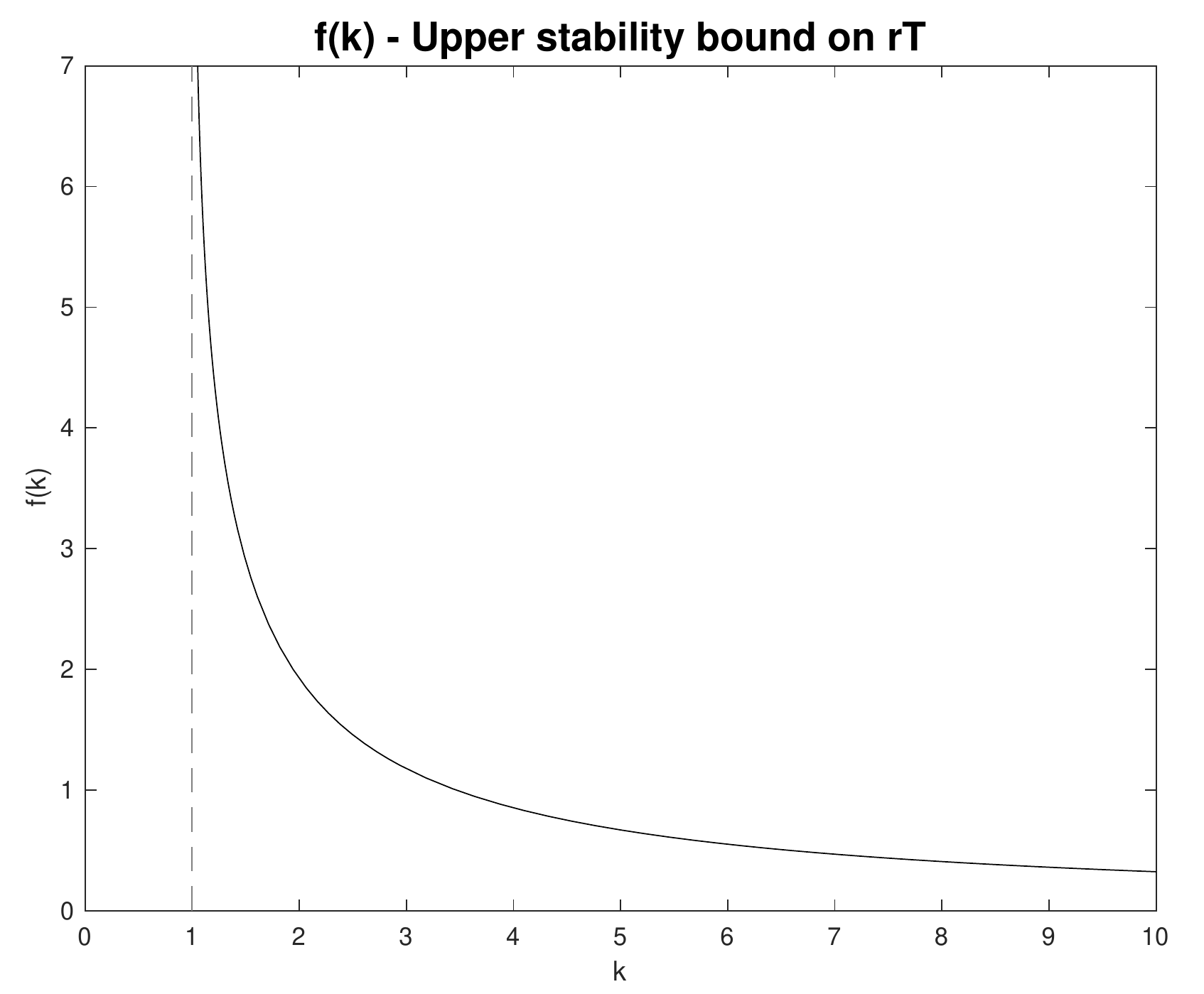}
\caption{
This function \protect{\eqref{rT_bound_1}} is monotonically decreasing since
$f'(k)  < 0$ for $k>1$, and $\displaystyle \lim_{k\to \infty} f(k) = 0.$
}
\label{figure3}
\end{figure}

Fig.~\ref{figure3} illustrates that as the delay $k$ increases, the range of possible $rT$ values which allow for optimal sustainable harvesting becomes smaller.
In the figure, the upper bound on sustainable $rT$ values 
\begin{equation*}
\label{rT_bound_1}
f(k) = - 2 \ln\bigg( 1 - 2 \cos\bigg(\frac{\pi k}{2k + 1} \bigg)\bigg),
\end{equation*}
is plotted and it is easy to see that as $k$ increases, $f(k)\to 0$. This parallels the theory of continuous delayed harvesting where a growing delay exhibits a destabilizing effect on the model.

The results of the paper can be summarized as follows:
\begin{enumerate}
\item
With delayed impulsive harvesting, we cannot guarantee positivity of a solution with non-negative and non-trivial initial conditions, moreover, extinction in finite time is possible.
\item
The delay does not influence the maximum yield but can significantly influence its sustainability.
\item
With sharp local stability conditions, the optimal solution associated with the maximum yield is locally asymptotically stable in both non-delay \cite{Zhang2003} and one-step-delay cases. 
For longer delays, there are bounds on $rT$ to attain MSY: the longer the delay is, the more frequent impulses should be. 
\end{enumerate}

For the choice of optimal harvesting, the increasing value of the intrinsic growth rate has negative influence on sustainability (see Figs.~\ref{figure1} and \ref{figure3}). This is well aligned with the enrichment paradox when increased productivity can lead to higher extinction risks.

Let us discuss time delays in the impulsive system. The real delay in information is $kT$ corresponding to $k$ cycles with $T$-time duration of each cycle. Assuming $r=1$ and applying the bound for $rT=T$ in \eqref{Cond:General}, we 
evaluate the delay $kT$. We get using L'H\^{o}spital's rule,
$$
\lim_{k \to +\infty} \bigg[ -2 k \ln\bigg( 1 - 2 \cos\bigg(\frac{\pi k}{2k + 1} \bigg)\bigg) \bigg] = \pi.
$$
Since $kf(k)$, with $f$ plotted in Fig.~\ref{figure3}, is monotone decreasing with the lower bound of $\pi$,
the delay $kT$ not exceeding $\pi$ (generally, $\pi/r$) will lead to MSY for any $k$. Note that $rT<\pi$, where $T$ is the delay, is a sharp stability constant in the autonomous continuous delay case.

This research extends the classical continuous harvesting models of Clark \cite{Clark1990} as well as impulsive harvesting models with no delay \cite{Mamdani2008,Zhang2003} which were generalized to impulsive harvesting with delay, outlining new instability effects. 
The  models closest to those considered in the present paper were explored in \cite{Church2017,Pei2006,Zhang2003}
and to some extent \cite{Mamdani2008}. In particular, optimal harvesting was discussed in \cite{Mamdani2008,Zhang2003}, with the same MSY being derived as in the present paper, however in the absence of the delay in the impulsive part, the MSY was unconditional. Delayed impulsive harvesting for the logistic equation and asymptotic stability of its positive periodic solutions was explored in \cite{Church2017,Pei2006}. Compared to the previous work, the main points of similarity and difference can be outlined as follows.
\begin{enumerate}
\item
We use the method of reducing an impulsive model to a difference equation, which was extensively applied in \cite{Mamdani2008,Zhang2003} 
and also used in \cite{Church2017}. However, the paper \cite{Church2017} avoids  high order difference equations by instead reducing an equation with delayed impulses to a non-delay system. 
The main technique applied to study impulsive models is Lyapunov-Razumikhin, leading to various global stability results.
But
the system considered in this paper is non-smooth (truncated) in the sense that the solution is assumed to be identically zero, once $N(t)=0$ for some $t$, 
and we are not aware of this method applied to such models.  
\item
In both \cite{Church2017,Pei2006}, the impulse delay is within one-cycle period i.e. less than or equal to $T^-$.
If the equation is non-delayed, and we use as a reference value the same $T$-period, this value at $t=(jT)^-$
can be expressed as $\beta(r,\tau)x((jT)^-)$, where $\beta$ is explicitly computed and dependent on the delay $\tau$, the intrinsic growth rate $r$ and, generally, on the growth law (logistic here). Then the situation is to some extent reduced to non-delayed impulses when local and global stability coincide. 
If the delay is equal to the impulse period $T=T^+$, collapse after a harvesting event is possible when the positive periodic solution is locally stable, see Fig.~\ref{figure2}.
To the best of our knowledge, the present paper is the first to report extinction in finite time and discrepancy between local and global stability for impulsively harvested equations. For applied models, this could possibly explain population collapses. On the other hand, \cite{Church2017} explored subtler properties, such as cycles and bifurcations.
\end{enumerate}

Let us dwell on the extension of this research, its modifications and generalizations.
It would be advantageous to get sufficient conditions on the initial values guaranteeing existence of a positive solution and to describe global asymptotic stability for such initial conditions. Note that this problem cannot be solved with modifying the per capita growth rate in the equation.  Even for the Gompertz growth model which is sustainable for any level of harvesting,
delayed impulsive harvesting can cause immediate extinction, once the stock decays quickly and is significantly overestimated during the harvesting event.

The approach of \cite{Mamdani2008} where a certain deduction is incorporated in each harvesting event, not contributing to the yield, can also be combined with delayed impulsive harvesting. It is expected that the maximum yield in this case will not change when impulses are delayed, but investigating its local or global attractivity is still an open question. The current research can be placed in the context of $T$-periodic or almost periodic $T(s)$.

Another minor extension will be including delays in the continuous part, for example,
switching from the logistic to the Hutchinson equation
$$
\frac{dN}{dt} = r N(t) \bigg(1 - \frac{N(t-jT)}{K_c} \bigg)
$$
for some $j \in {\mathbb N}$, with the same impulsive conditions. 
Here again, the maximum yield will not change, but an additional delay is expected also to contribute to instability of the model.

Including a delay within impulsive harvesting allowed us to observe 
quite realistic effects, such as population collapse even when the model parameters predicted local stability of the unique positive periodic solution.

Certainly incorporating impulsive delayed harvesting in non-autonomous or structured population models will lead to richer dynamics, in particular, considering
\begin{itemize}
\item
non-autonomous models with variable parameters, in particular, considering periodic and almost periodic solutions (similarly to \cite{Yao}),
stability, bifurcations, as well as optimal harvesting policies and their sustainability;
\item
predator-prey systems where either one of the species or both are subject to harvesting, structured and Lotka-Volterra models which - without delay in impulsive harvesting or an impulsive system with continuous harvesting - 
are well-studied areas, see, for example, \cite{Kar2003,Jiao2019,Liu2020,Meng2016,Quan2021,Sharma2022,xiao_2015,Wang2021,Wang2022,Zhang2020};
\item
fractional derivative models and spatially distributed populations, as well as systems within fractional settings to describe spatial interactions.    
\end{itemize}

\section*{Acknowledgment}

The authors are grateful to the anonymous reviewer whose thoughtful comments significantly contributed to the quality of presentation.
The authors also acknowledge the support of NSERC, the grant RGPIN-2020-03934.

\end{document}